\documentclass{elsarticle}

\usepackage{amsmath,enumerate,amsrefs,amssymb,amsthm,nicefrac}

\newtheorem{theorem}{Theorem}[section]
\newtheorem{lemma}[theorem]{Lemma}

\newtheorem{fact}[theorem]{Fact}
\newtheorem{question}[theorem]{Question}

\newtheorem*{acknowledgement*}{Acknowledgement}
\newtheorem*{theorem*}{Theorem}

\theoremstyle{definition}
\newtheorem{definition}[theorem]{Definition}

\theoremstyle{remark}
\newtheorem{remark}[theorem]{Remark}

\numberwithin{equation}{section}

\newcommand{\eps}{\varepsilon}

\newcommand{\R}{\mathbb R}

\newcommand{\B}{\mathcal B}
\newcommand{\C}{\mathcal C}

\def\XXint#1#2#3{{\setbox0=\hbox{$#1{#2#3}{\int}$}
		\vcenter{\hbox{$#2#3$}}\kern-.5\wd0}}

\DeclareMathOperator{\dist}{dist}

\DeclareMathOperator{\diam}{diam}

\DeclareMathOperator{\sign}{sign}

\begin{document}
	
	\title{Zygmund graphs are thin for doubling measures}
	\author[rvt]{Claudio A. DiMarco}
	\address[rvt]{1000 E. Henrietta Rd., Mathematics Department, Monroe Community College, Rochester, NY 14623, USA}
	\ead{cdimarco2@monroecc.edu}
	\date{\today}

\begin{abstract}
The Zygmund functions form an intermediate class between Lipschitz and H\"older functions; their second order divided differences are uniformly bounded.  It is well known that for $d \geq 1$ the graph of any Lipschitz function $f:\R^d \rightarrow \R$ is thin for doubling measures, and we extend this result to the Zygmund class.
\end{abstract}

\maketitle

\section{Introduction}
Certain subsets $E$ of a metric space $X$ are ``thin" or ``fat" for doubling measures.  Wu initiated the discussion of thinness for doubling measures on $\R$ in \cite{Wu}, wherein the relationship between thinness and ``porosity" was also investigated. The results in \cite{Wu} employ a generalized notion of porosity which easily implies thinness for doubling measures, as Chen and Wen noted in \cite[page 1]{Chen}.

In \cite{Ojala2} the authors considered the case when $X$ is uniformly perfect and $E \subset X$ has the property $\mu(E) = 0$ for every doubling measure $\mu$ on $X$, or the property $\mu(E) > 0$ for all doubling measures. Such subsets are called thin and fat, respectively.  Ojala et al provided sufficient conditions for certain ``cut-out sets" being thin or fat, and proved that $E \subset X$ is thin if and only if $E$ is quasisymmetrically null, i.e. $\mathcal{H}^q(f(E)) = 0$ whenever $f:X \rightarrow Y$ is quasisymmetric and $Y$ is Ahlfors $q$-regular, where $\mathcal{H}^q$ is $q$-dimensional Hausdorff measure \cite{Ojala2}.  

It is interesting to consider the special case when $E \subset \R^{d+1}$ is the graph of a continuous function.  This is an area of recent interest in which little is understood.  For a while, an open question in this area was whether every rectifiable curve in the plane is thin for doubling measures. Surprisingly, this turned out not to be the case (see \cite{Garnett}).  However, if attention is restricted to so called ``isotropic" doubling measures on $\R^{d+1}$, then indeed the graph of any continuous function $f:[0,1]^d \rightarrow [0,1]$ has measure zero \cite{Chen}.  This is easily extended to include all continuous $f: \R^d \rightarrow \R$.

Recall that a bounded function $f: \R^d \rightarrow \R$ is H\"older continuous if there are $0<\alpha \leq 1$ and $C>0$ such that $|f(x) - f(y)| \leq C|x-y|^{\alpha}$ for all $x,y \in \R^d,$ and is called Lipschitz if there is $L>0$ such that $|f(x) - f(y)| \leq L|x-y|$ for all $x,y \in \R^d$.  In \cite{Ojala} the authors showed that even a graph, $y=f(x)$, of a continuous function, need not be thin in general.  They remark that there are few known results in the positive direction, i.e. sufficient conditions for thinness. In particular, it is known that the graph of a Lipschitz function is thin, but whether this property extends to H\"older functions remains unknown \cite[Question 1.2]{Ojala}.

In 1876 Weierstrass introduced the now familiar family of continuous but nowhere differentiable functions
\[ 
f(x) = \sum_{n = 0}^{\infty} a^n \cos(b^n \pi x), ~~x\in \R 
\]
where $0 < a < 1$ and $b$ is an odd integer with $ab > 1 + 3\pi/2$.  Over time several improvements were made on the lower bound $ab > 1 + 3\pi/2$, and in 1916 Hardy proved that any function of the form 
\[ \sum_{n = 0}^{\infty} a^n \cos(b^n \pi x) ~~\text{or}~~ \sum_{n = 0}^{\infty} a^n \sin(b^n \pi x) \]
is nowhere differentiable whenever $ab \geq 1$; see \cite{Hardy}.  The critical case $ab = 1$ produces 
\[ f(x) = \sum_{n = 0}^{\infty} b^{-n} \cos(b^n \pi x), ~~x\in \R, \]
which can be characterized by the behavior of its second divided differences:
\begin{definition}\label{zyg_class}
	The \textbf{Zygmund class} $\Lambda_*(\mathbb{R}^d)$ is the space of bounded continuous functions $f: \mathbb{R}^d \rightarrow \mathbb{R}$ for which 
	\begin{equation}\label{sup_2nd_diff}
	\|f\|_* = \sup \left\{ \frac{|f(x + h) + f(x - h) - 2f(x)|}{|h|} : x,h \in \mathbb{R}^d \right\} < \infty.
	\end{equation}
\end{definition}
These spaces were introduced by Zygmund when he noticed that the conjugate function of a Lipschitz function in the unit circle does not need to be Lipschitz, but it is in the Zygmund class \cite{Zygmund, Donaire}.

For any $0<\alpha \leq 1$, denote the H\"older class by $\Lambda_{\alpha}(\R^d)$ and notice that the Lipschitz class is obtained for $\alpha = 1$.  It is well known that if $0 < \alpha < 1$ then $\Lambda_1(\R^d) \subset \Lambda_*(\R^d) \subset \Lambda_{\alpha}(\R^d)$, and because Lipschitz graphs $E \subset \R^{d+1}$ are thin for doubling measures, it is natural to extend this result to the Zygmund class (which is the main result):
\newtheorem*{main_result}{Theorem \ref{main_result}}
\begin{main_result}
	If $f: \R^d \rightarrow \mathbb{R}$ is a Zygmund function, then the graph of $f$ is thin for doubling measures.  That is, if $\mu$ is a doubling measure $\mu$ on $\mathbb{R}^{d+1}$ and $E$ is the graph of $f$, then $\mu(E) = 0.$
\end{main_result}

There is a noticeable relationship between Definition \ref{zyg_class} and that of differentiability, and recent progress has been made in relating the two notions.  In \cite{Donaire} the authors showed that for any $f\in\Lambda_*(\R^d)$, the set of points at which the first divided differences are bounded has Hausdorff dimension at least one. They also proved that any function in the small Zygmund class (see \cite{Donaire} for the definition) is differentiable on a set of Hausdorff dimension at least one.  This result contrasts with the large Zygmund class because a function $f\in \Lambda_*(\R^d)$ may be nowhere differentiable.

\section{Preliminaries} 
For a metric space $(X,d),$ write $A_{\delta}$ for the open $\delta$-neighborhood of $A\subset X.$  For $x = (x_1, x_2, \dots, x_n) \in \mathbb{R}^n$ the quantity $|x| = \|x\|$ is the norm of $x$ and $x[i]$ is the $i$th coordinate of $x$.  For $1 \leq i \leq n$ we write $e_i \in \R^n$ for the standard basis vector defined by $e_i[i]=1$ and $e_i[j]=0$ whenever $j\neq i$.  As usual $\lfloor p \rfloor$ and $\lceil p \rceil$ are the floor and ceiling of $p \in \R$ respectively.

All cubes $Q\subset\R^d$ are assumed to be closed unless otherwise noted, and we use $B(x,r)$ and $\bar{B}(x,r)$ for open and closed balls of radius $r$ respectively.  Also $B_{d+1}(x,r)$ is used to denote open balls in $\R^{d+1}$ to avoid confusion in the context of a function $f: \R^d \rightarrow \R.$  For an open ball $B = B(x,r)$, for any $s > 0$ we write $sB = B(x, sr).$  We will also make use of the the function
\[ \sign(x) = \begin{cases}
	1, & x \geq 0 \\
	-1, & x < 0. \\
\end{cases}\]

A metric space is called doubling if there is $C_1 \geq 1$ so that every set of diameter $d$ in the space can be covered by $C_1$ sets of diameter at most $d/2$ \cite[page 81]{Heinonen}.  Every complete doubling metric space supports a doubling measure \cite{Luukkainen}.  The following definition can be found in \cite[page 3]{Heinonen}:
\begin{definition}\label{doubling_def}
	A Borel measure $\mu$ on a metric space $(X,d)$ is called \textbf{doubling} if there is $C>0$ such that $\mu(\bar{B}(x, r)) \leq C \mu(\bar{B}(x, r/2))$ for all nonempty closed balls $\bar{B}(x,r)$.
\end{definition}
A doubling metric space that supports a doubling measure $\mu$ is often denoted $(X, d, \mu)$.  Given such a space, the doubling property  of $\mu$ ensures that the measure is not disproportionately concentrated around any particular point $x\in X$.
\begin{definition}
	A subset $E$ of a metric space $(X,d)$ is \textbf{thin} if $\mu(E)=0$ for every doubling measure $\mu$ on $X$.
\end{definition}

In \cite{Jonsson} the authors defined so-called Lipschitz spaces $L_{\alpha}(F)$ on closed subsets $F \subset \R^d$ for all $\alpha > 0$.  An integral part of this definition involves polynomial approximations $P$ to $f$, where the degree of $P$ is at most $\lfloor \alpha \rfloor$.  The case $\alpha = 1$ corresponds to the Zygmund class (i.e. $L_1(F) = \Lambda_*(F)$), in which case $\lfloor \alpha \rfloor = 1$ furnishes degree one polynomial (i.e. affine) approximations to $f$ \cite[chapter III, section 2.1]{Jonsson}.  These polynomial approximations are valid on cubes in $F = \R^d$, and therefore on balls via restriction.  We use the following simplified formulation because it is easily applied to the problem at hand.
\begin{fact}\label{plane_approx}
	For any Zygmund function $f \in \Lambda_*(\R^d)$ and any ball $B = B(x_0, r) \subset \R^d$, there is a degree-1 polynomial $P$ such that 
	\begin{equation}\label{M_def}
		\sup_{x \in B} |f(x)-P(x)| \le Mr,
	\end{equation}
	with a constant $M$ independent of both $x_0$ and $r$.  Moreover, this property characterizes Zygmund functions (\cite[Proposition 3, page 54]{Jonsson-Wallin}, \cite[page 155]{Jonsson}).
\end{fact}

Using the characterization of Zygmund functions in Fact \ref{plane_approx}, Lemma \ref{P_balls} shows that if the gradient of $P$ is not too large, then above each point $z\in E \subset \R^{d+1}$ there is a ball of suitable size which lies above $E$. If $\nabla P$ is large, then there is a ball in the direction of $\nabla P$ where $E$ is above a ball in $\R^{d+1}$ of suitable size.

\begin{lemma}\label{P_balls}
	Suppose $f \in \Lambda_*(\R^d)$, let $E \subset \R^{d+1}$ denote the graph of $f$, and fix $z_0  = (x_0, f(x_0)) \in E$.  For any $\delta > 0$ there is an open ball $B = B(z_0, r)\subset E_{\delta}$ and a corresponding ball $B(z, r) \subset (E_{\delta} \setminus E)$.  Moreover $r>0$ may be chosen arbitrarily small, and there is $M>1$ such that one of the following conditions holds for each $B$:
	\begin{enumerate}[(i)]
		\item  $z[i] = {z_0}[i]$ for $i < d+1$ and $|z - z_0| = (2M+1)r$ ~~\textbf{or} 
		\item  $|z - z_0| \leq (5M+1)r$ where the constant $K$ depends on $M$ and is independent of $z$ and $r$.
	\end{enumerate}
\end{lemma}
\begin{proof}
	Let $\delta > 0$, fix $z_0  = (x_0, f(x_0)) \in E$, and let $M>1$ as in Remark \ref{plane_approx}.  There are two cases:
	\begin{enumerate}[(i)]
		\item There is $\beta > 0$ such that for every open ball $B = B(x_0, \eps) \subset E_{\delta}$ with $\eps < \beta$, the corresponding affine function $P_B$ as in Remark \ref{plane_approx} has $|\nabla P_B| \leq 1$.  For any such $B$, writing $P_B(x) = \sum_{i = 1}^d a_i x_i$ with $x = (x_1, x_2, \dots, x_d)$ shows 
		\begin{equation}
		\begin{split}\label{diff_in_P}
			|P_B(x) - P_B(x_0)|^2 &= |P_B(x - x_0)|^2 \\
			&\leq \left( \sum_{i = 1}^d |a_i| \left|x_i - x_0[i] \right| \right)^2 \\
			&\leq \sum_{i = 1}^{d} a_i^2 \left|x_i - x_0[i] \right|^2 \\
			&\leq \eps^2 \sum_{i = 1}^{d} a_i^2 \\
			&= \eps^2 |\nabla P|^2. \\
		\end{split}
		\end{equation}
		Choose $r>0$ small enough that both $B_{d+1} = B(z_0, r)$ and $B_{d+1}' = B(z_0 + (2M+2)r e_{d+1}, r)$ lie in $E_{\delta}$.  (Notice that $M>1$ implies these two balls are disjoint.)  Because $|\nabla P_B| \leq 1$, observe that for all $x \in B = B(x_0, r)$, in light of equation \eqref{M_def},
		\[
		\begin{split}
			|f(x) - f(x_0)| &\leq |f(x) - P_B (x)| + |P_B(x) - f(x_0)| \\
			&\leq Mr + |f(x_0) - P_B(x_0)| + |P_B(x_0) - P_B(x)| \\
			&\leq Mr + Mr + |\nabla P| r \\
			&\leq (2M+1)r,
		\end{split}
		\]
		from which it follows that $B_{d+1}' \cap E = \varnothing$, i.e. $B_{d+1}$ lies entirely above $E$ because \[ \dist(z_0, B_{d+1}') = (2M+1)r.\]
		
		\item There is a sequence of open balls $B_n = B(x_0, r_n)$ such that $r_n \rightarrow 0$ and  $|\nabla P_{B_n}| >1$.
		
		Choose $r>0$ small enough that both the following conditions are satisfied: 
		\begin{enumerate}[(a)]
			\item $B = B(x_0, r)$ has $P_B(x) = \sum_{i=0}^d, a_i x[i]$ with $|\nabla P_B| > 1$, and 
			\item $(5M+2)r < \delta$ so that $B(z, (5M+2)r) \subset  E_{\delta}$ for every $z \in E$.
		\end{enumerate}
		Put
		\begin{equation}\label{r_prime_def}
		r' = \min \left( r/2, \frac{Mr}{|\nabla P_B|} \right).
		\end{equation}
		and let $x_0' = x_0 + 2r' u$ where $u = \frac{\nabla P_B}{|\nabla P_B|}$.  Observe that for all $x \in B(x_0', r')$, equations \eqref{M_def},\eqref{diff_in_P} and \eqref{r_prime_def} provide
		\begin{equation}
		\begin{split}\label{main_eqn_2}
			|f(x) - f(x_0')| &\leq |f(x) - P_B (x)| + |P_B(x) - P_B(x_0')| + |P_B(x_0') - f(x_0')| \\
			&\leq 3Mr.
		\end{split}
		\end{equation}
		Let $z_0' = (x_0', f(x_0')) \in \R^{d+1}$ and define 
		\[z_0'' = z_0 + (2r'u, \sign(f(x_0) - f(x_0'))(3M+1)r), \] so that $B_{d+1}' = B(z_0'', r')$ resides either entirely below or entirely above $E$.  Indeed,
		\begin{equation}
			\begin{split}\label{magic_ball}
				|z_0 - z_0''| &\leq 2r' + (3M+1)r \\
				&\leq 2Mr + (3M+1)r \\
				&= (5M+1)r.
			\end{split}
		\end{equation}
		It follows from condition (b) and equation \eqref{main_eqn_2} that $B_{d+1}' \subset (E_{\delta} \setminus E)$ as required.
	\end{enumerate}
\end{proof}

\begin{remark}\label{5M_remark}
	The constant $2M+1$ in the conclusion of Lemma \ref{P_balls} may safely be replaced by $5M+1$ in case the ball $B_{d+1}'$ lies directly above $z_0 \in E$.  To this end, it suffices to choose $r$ small enough that $(5M+2)r < \delta$.  This observation simplifies the proof of Theorem \ref{main_result} by combining two cases.
\end{remark}

\begin{theorem}\label{main_result}
	If $f: \R^d \rightarrow \mathbb{R}$ is a Zygmund function, then the graph of $f$ is thin for doubling measures.  That is, if $\mu$ is a doubling measure $\mu$ on $\mathbb{R}^{d+1}$ and $E$ is the graph of $f$, then $\mu(E) = 0.$
\end{theorem}
\begin{proof}
	Let $\mu$ be a doubling measure on $\mathbb{R}^{d+1}$ with doubling constant $C,$ define $M$ as in Fact \ref{plane_approx}, and let $\eps >0.$  Put 
	\begin{equation}\label{K_def}
	K = \lceil 5M+1 \rceil + 1 ~~\text{and}~~ p = \left\lceil \frac{\ln(5K)}{\ln(2)} \right\rceil.
	\end{equation}
	Because $\bigcap_{\delta>0}(E_{\delta}\setminus E)=\varnothing,$ there is $\delta>0$ such that $\mu(E_{\delta}\setminus E) < \eps/(C^p)$.  Choose a covering $\B$ of $E$ by open balls as in Lemma \ref{P_balls}.
		
	Consider the collection of open balls $\hat{\B}$ obtained by translating each $B = B(z,r) \in \B$ to one of its offset balls $\hat{B} = B(w, r) \subset (E_{\delta} \setminus E)$, where $|z-w| = (5M+1)r$ as in Lemma \ref{P_balls} and Remark \ref{5M_remark}.  By construction, each $z \in E$ is the center of some  $B \in \B$, so $z \in K\hat{B}$ by \eqref{K_def} and hence
	\begin{equation}\label{3B_shifted}
		E \subset \bigcup_{\hat{B} \in \hat{\B}} K\hat{B}.
	\end{equation}
	The Vitali covering lemma (see e.g. \cite[page 3]{Heinonen}) provides a countable disjoint subcollection $\hat{\C} \subset \hat{\B}$ such that
	\begin{equation}\label{15B_shifted}
		\bigcup_{\hat{B} \in \hat{\B}} K\hat{B} \subset \bigcup_{\hat{B} \in \hat{\C}} (5K) \hat{B}. 
	\end{equation}
	The fact that $\hat{B} \subset (E_{\delta} \setminus E)$ for every $\hat{B}$, combined with the doubling property of $\mu$, equations \eqref{K_def},\eqref{3B_shifted},\eqref{15B_shifted}, and the fact that $\hat{C}$ is disjoint, yield
	\begin{equation}
		\begin{split} \label{main_eqn}
			\mu(E) &\leq \mu \left( \bigcup_{\hat{B} \in \hat{\C}} (5K)\hat{B} \right) \\
			&\leq \sum_{\hat{B} \in \hat{\C}} \mu((5K) \hat{B}) \\
			&\leq \sum_{\hat{B} \in \hat{\C}} C^p \mu(\hat{B}) \\
			&= C^p \mu \left( \bigcup_{\hat{B} \in \hat{\C}} \hat{B} \right) \\
			&\leq C^p \mu(E_{\delta} \setminus E) \\
			&< C^p (\eps/(C^p)) \\
			&= \eps. \qedhere
		\end{split}
	\end{equation}
\end{proof}

The relationship between the graph $E$ and the collection of offset open balls $\hat{B} \subset (E_{\delta} \setminus E)$ used in the proof of Theorem \ref{main_result} raises the question of porosity.  The following definition can be found in \cite{Chen}.
\begin{definition}\label{porous_def}
	A subset $E \subset \R^d$ is called porous if there exists $a \in (0,1)$ such that for any ball $B(x, r)$, there is a ball $B(y, a r) \subset B(x, r)$ satisfies 
	\[B(y, a r) \cap E = \varnothing.\]
\end{definition}
\begin{remark}
	The value $r>0$ employed in the proof of Lemma \ref{P_balls} is independent of the chosen reference point $z_0 \in E$.  However, the smaller radius $0<r'<r$, which produces the ball $B(y, r') \subset (B(z_0, r) \setminus E),$ depends on $z_0$.  Thus the question of porosity remains open.
\end{remark}

\begin{question}
	Given any Zygmund function $f:\R^d \rightarrow \R$, is the graph porous in $\R^{d+1}$?
\end{question}

\begin{remark}
	Because $\R^n$ is uniformly perfect, it follows from Theorem \ref{main_result} above and Proposition 1.2 in \cite{Ojala2} that the graph $E$ of any Zygmund function is quasisymmetrically null.  On the other hand, Theorem 1.1 in \cite{Ojala} guarantees that $E \subset X$ is thin if there are closed balls 
	\[\{\diam \bar{B}_i\}_{i = 1}^{\infty} \in \ell^0 := \bigcap_{p>0} \ell^p\] such that $E = X \setminus \bigcup \bar{B}_i$.  This leads to the following question of the converse.
\end{remark}
\begin{question}
	Suppose $f \in \Lambda_*([a,b]^d)$, denote its graph $E \subset \R^{d+1}$.  Are there $M>0$ and a sequence of closed balls $\bar{B}_i \subset \overline{E_M}$ such that $\{\diam \bar{B}_i\}_1^{\infty} \in \ell^0$ and $E = \overline{E_M} \setminus \bigcup \bar{B}_i$?
\end{question}

Considering the thinness of Zygmund graphs in euclidean space, hope remains for an affirmative answer to the following question posed in \cite{Ojala}.
\begin{question}
	Is the graph of every H\"older function $f \in \Lambda_{\alpha}(\R^d)$ thin for doubling measures?
\end{question}

\begin{acknowledgement*}
	The author thanks Leonid Kovalev for suggesting references \cite{Jonsson}, \cite{Jonsson-Wallin} and \cite{Wu}.
\end{acknowledgement*}

\end{document}